\documentclass[12pt, a4paper]{article}
\usepackage[centertags]{amsmath}
\usepackage{amsfonts}
\usepackage{amssymb}
\usepackage{amsthm}
\usepackage{amscd}
\usepackage{euscript}
\usepackage{graphicx}
\usepackage{hyperref}
\usepackage[matrix,arrow,curve]{xy}
\usepackage{verbatim} 
\usepackage{cases}
\usepackage{marginnote}
\usepackage{stmaryrd}

\newtheorem{thm}{Theorem}[section]
\newtheorem{cor}[thm]{Corollary}
\newtheorem{lem}[thm]{Lemma}
\newtheorem{prop}[thm]{Proposition}

\theoremstyle{definition}
\newtheorem{defn}[thm]{Definition}
\newtheorem{rem}[thm]{Remark}
\newtheorem{ex}[thm]{Example}

\newcommand{\hK}{hyperK\"{a}hler }

\newcommand{\wrt}{with respect to }
\newcommand{\Kvf}{{Killing vector field }}
\newcommand{\kahler}{{K{\"a}hler }}

\newcommand{\TP}{{Taubes-Pid\-stry\-gach }}

\newcommand{\compcent}[1]{\vcenter{\hbox{$#1\circ$}}}
\newcommand{\comp}{\mathbin{\mathchoice {\compcent\scriptstyle}{\compcent\scriptstyle}
{\compcent\scriptscriptstyle}{\compcent\scriptscriptstyle}}}

\newcommand{\set}[1]{\left\{#1\right\}}
\newcommand{\e}{\varepsilon}

\newcommand{\bM}{{\mathbb M}}
\newcommand{\BundleLieAlgG}{{\mathbb L}}
\newcommand{\om}{{\omega}}
\newcommand{\Om}{{\Omega}}
\renewcommand{\c}[1]{{#1}_{\ssst{\mathbb C}}}
\renewcommand{\hom}[3]{ {Hom_{\mathbb{#1}} \left ( {#2} , {#3}\right )} }
\renewcommand{\Im}{{\text{\rmfamily\upshape Im} \,}}
\renewcommand{\Re}{{\text{\rmfamily\upshape Re} \,}}
\newcommand{\ImH}{{\Im\mathbb H}}

\newcommand{\hkred}{{/\!\! /\!\! /}}

\newcommand{\duzhky}[1]{{\left ( #1\right )}}
\newcommand{\ssst}[1]{{\scriptscriptstyle #1}}
\DeclareMathOperator{\TOmega}{T_{\Omega}}
\DeclareMathOperator{\dirac}{{\mathcal D}}
\newcommand{\overlineshort}[1]{\ensuremath{\hspace{1pt}\overline{\hspace{-1pt}#1\hspace{-1pt}}}\hspace{1pt}} 
 
\begin{document}

\title{Gauge theory, calibrated geometry and harmonic spinors}%
\author{Andriy Haydys\\%
\textit{University of Bielefeld}}%



\date{January 14, 2011}
\maketitle
\begin{abstract}
In this paper connections between different gauge-theoretical problems in high and low dimensions are established. In particular it is shown that higher dimensional anti-self-duality equations on total spaces of spinor bundles over low dimensional manifolds can be interpreted as Taubes-Pidstrygach's generalization of the Seiberg-Witten equations. By collapsing each fibre of the spinor bundle to a point, solutions of the Taubes-Pidstrygach equations are related to generalized harmonic spinors.  This approach is also generalized for arbitrary fibrations (without singular fibres) compatible with an appropriate calibration.
\end{abstract}

\section{Introduction}

The basic set-up of low dimensional gauge theory have been generalized to higher dimensions by Donaldson and  Thomas~\cite{DonaldsonThomas:98}. Such generalization requires a suitable geometric structure on the base higher dimensional manifold, e.g., a metric with holonomy $G_2,\ Spin(7)$ or $SU(n)$. In this paper we consider gauge theory based on the  $Spin(7)$-structure mainly because of applications we keep in mind, but we could equally well start with $G_2$-- or $SU(n)$--structures.

A Riemannian eight-manifold $W$ with holonomy $Spin(7)$ is endowed with a particular closed 4-form $\Om$ called Cayley form, which is a calibration~\cite{HarveyLawson:82}. Vice versa, a Cayley calibration determines a metric with holonomy $Spin(7)$. The 4-form $\Om$ induces the splitting of $\Lambda^2T^*W$ into two subbundles $\Lambda^2_+T^*W$ and  $\Lambda^2_-T^*W$ of rank $7$ and $21$ respectively. We say that a connection $A$ is anti-self-dual (or that $A$ is a $Spin(7)$-instanton) if the self-dual part of the curvature $F_A^+$ vanishes.

%

A particularly important role in our approach is played by eight-manifolds with a structure of a Cayley fibration, i.e. a map $\rho\colon W^8\rightarrow X^4$ such that each fibre of $\rho$ is an $\Om$-calibrated submanifold of $W$. The first nontrivial example of a Cayley fibration has been constructed in~\cite{BryantSalamon:89}. In the construction of Bryant and Salamon $W$ is the total space of the spinor bundle over the sphere $S^4=X$ and $\rho$ is the natural projection.  If $X$ is an arbitrary spin manifold, then on the total space of the spinor bundle the Cayley 4-form $\Om$ still exists, but it is not closed in general. Nevertheless, the asd condition still makes sense and gives rise to an elliptic problem.

One motivation to study the equations on the total spaces of spinor bundles over general four-manifolds is as follows. Recall that the key ingredients in Tian's construction of the compactified moduli space of higher dimensional instantons are (possibly singular) calibrated submanifolds. Assume $X^4\hookrightarrow W^8$ is a smooth calibrated submanifold (Cayley submanifold). Then, by the result of McLean~\cite{McLean:98}, the normal bundle of $X$ is isomorphic to the (twisted) spinor bundle of $X$. Therefore we hope that our computation will be useful in a detailed study of the boundary of the compactified moduli space. We refer to~\cite[Section~6]{DonaldsonSegal:09} for more details on this issue.

Another motivation comes from low dimensional topology. Suppose we are granted a construction that associates a $Spin(7)$-manifold $W^8_X$ to each smooth four-manifold  $X$ (possibly equipped with an additional structure). Then, by counting $Spin(7)$-instantons on $W_X$ we should get an invariant of $X$. For instance, we can associate to each spin four-manifold the total space of its spinor bundle as mentioned above. Then  Theorem~\ref{Thm_ASDarePSW} (the main result of this paper) essentially states that counting instantons on the spinor bundle of $X$ is equivalent to counting Taubes-Pidstrygach monopoles~\cite{Taubes:99,Pidstrygach:04} on $X$.


There are a few problems with the above approach. One is that the Cayley form $\Om$ on $W_X$ is not always closed as already mentioned above. However, the theory of the asd equations can still be developed if the closedness condition of $\Om$ is suitably weakened, see for example~\cite[Thm.\,6.1.3]{Tian:00} and~\cite[Sect.3]{DonaldsonSegal:09}.  In any case, the closedness of $\Om$ is \emph{not} essential for our computations.

Another problem is as follows. The main idea of the Taubes-Pidstrygach generalization of the Seiberg-Witten equations is to replace the fibre $\mathbb C^2$ of the spinor bundle with an arbitrary \hK manifold $M$ equipped with suitable symmetries. In our case $M$ happens to be infinite-dimensional. There is however some evidence that the resulting low-dimensional gauge theory can be phrased in terms of finite-dimensional target spaces. This issue is briefly discussed in Section~\ref{Sect_ConcludingRemarks}.

\medskip

The paper is organized as follows. In introductory Section~\ref{Sect_ToyModel} main ideas in the simplest case of the flat base space are briefly sketched. In Section~\ref{Sect_Preliminaries} some basic definitions are recalled and the construction of the Bryant-Salamon precalibration~\cite{BryantSalamon:89} on the total space of the spinor bundle $\mathbb W^+$ over a four-manifold $X$ is reviewed.

Section~\ref{Sect_ASDOnEightMnflds} is the core of the paper. In subsection~\ref{Subsec_PerturbedTPeqns} we study formal aspects of the adiabatic limit for \TP monopoles. In subsection~\ref{Subsect_ASDasPSW} we show that asd equations on $\mathbb W^+$ can be interpreted as \TP equations for a suitable choice of the target space $M$ (Theorem~\ref{Thm_ASDarePSW}). By collapsing each fibre of $\mathbb W^+$ to a point we show that the moduli space of $Spin(7)$-instantons corresponds to the space of generalized harmonic spinors~\cite{Haydys_ahol:08}, whose target space is the moduli space of framed four-dimensional instantons (Corollary~\ref{Cor_InstantonsAndHarmSpinors}). This is a variant of the adiabatic limit reduction outlined in~\cite{DonaldsonThomas:98}. It is worth to note that  Corollary~\ref{Cor_InstantonsAndHarmSpinors} can be proven without the \TP construction, but in the authors opinion this statement is best understood from this more abstract point of view at least if one is interested in applications to low dimensional topology.

We do not discuss analytic aspects of the adiabatic limit in this paper for several reasons. The major reason is that a prerequisite for the analytic part of the  proof is a completion of the space of nonlinear harmonic spinors, which is yet to be constructed. This problem also arises in the recent paper~\cite{Salamon_hKFloer:08}. The case of the Cartesian product of two \hK four-manifolds have been studied in~\cite{Chen:99}.


In Section~\ref{Sect_InstantonsAndCayleyFibr} we show how our previous results  modify in the case of an arbitrary Cayley fibration without singular fibres.


\section{A toy model}\label{Sect_ToyModel}

In this section basic ideas in the simplest case of the flat space are outlined. More details on computations are given in the subsequent sections.

Think of the flat space $\mathbb R^8$ equipped with a translation-invariant Cayley 4-form $\Om$ (see~\eqref{Eq_Standard4Form}) as the spinor bundle of the flat four-manifold $X=\mathbb R^4:\ \mathbb R^8=X\times W^+$.
%
A connection $A$ invariant \wrt the $W^+$-directions on the trivial $G$-bundle consists of a connection $a$ on $\underline G\rightarrow X$ and a Higgs field $\Phi$, which is a section of the trivial bundle with fibres $W^+\otimes\mathfrak g_{\mathbb C}$, where $\mathfrak g=Lie(G)$. Donaldson and Thomas~\cite{DonaldsonThomas:98} observe that $A$ is a $Spin(7)$-instanton iff the following equations hold%
\begin{equation}\label{Eq_4DHitchinEqns}
       \dirac_a\Phi=0,\quad
       F_a^+=[\Phi, \Phi^*].
\end{equation}
Here $\dirac_a$ is the Dirac operator on $X$ coupled to the connection $a$, the bracket in the second equation is a combination of the Lie-bracket and the map $W^+\otimes\overlineshort W^+\rightarrow \Lambda^2_+T^*X$. It is worth pointing out the striking similarity between equations~(\ref{Eq_4DHitchinEqns}) and the renowned Seiberg-Witten equations.

%

It turns out that the above observation fits into a much wider picture. Namely, let $V$ denote the tangent space of $X$ at a fixed point (the origin, say). Observe that both $V$ and $W^+$ are equipped with the quaternionic structures and therefore we can naturally identify both $\Lambda^2_+ V^*$ and $\Lambda^2_+ (W^+)^*$ with $\mathbb R^3$. Decompose a 2-form $\om$ on $X\times W^+$ into its K\"unneth-type components, $\om=\om_{2,0}+\om_{1,1}+\om_{0,2}$, and think of $\om_{p,q}$ as a function on $\mathbb R^8$ with values in $\Lambda^p V^*\otimes\Lambda^q(W^+)^*$. Then a computation shows that $\om^+=\frac 14(\om-*(\Om\wedge\om))$ vanishes iff
\begin{equation}\label{Eq_ASDinComponents}
       Cl(\om_{1,1})=0\quad\text{and}\quad
       \om_{2,0}^+=\om_{0,2}^+.
\end{equation}
Here $Cl\colon V^*\otimes (W^+)^*\cong V\otimes W^+\rightarrow W^-$ is the standard Clifford multiplication in dimension four, $\om_{2,0}^+\in C^\infty(\mathbb R^8; \Lambda^2_+V^*)\cong C^\infty(\mathbb R^8; \mathbb R^3)$, and $\om_{0,2}^+$ is interpreted similarly.

Likewise, decompose a connection $A$ on the trivial $G$-bundle 
into its K\"unneth-type components: %
\begin{equation*}
 A=a+b, \quad a\in C^\infty\duzhky{\mathbb R^8;\, (W^+)^*\otimes\mathfrak g},\  b\in C^\infty\duzhky{\mathbb R^8;\, V^*\otimes\mathfrak g}.
\end{equation*}
 Then  it follows from~(\ref{Eq_ASDinComponents})  that $A$ is a $Spin(7)$-instanton iff the following equations hold%
\begin{equation}\label{Eq_ToySpin7Instantons}
       Cl\bigl (d_X a +d_{W^+}b+ [b,a]\bigr )=0,\qquad
       F_b^+=F_a^+,
\end{equation}
where $a$ and $b$ can be thought of as families of connections on trivial $G$-bundles over $W^+$ and $X$ respectively.

\medskip

Putting away $Spin(7)$-instantons for a while, we describe the generalization of the Seiberg-Witten equations due to Taubes~\cite{Taubes:99} and Pidstrygach~\cite{Pidstrygach:04} in a special case. Namely, for a quaternion Hermitian vector space $E$ assume an action of a Lie group $\mathcal G$ preserving its quaternion Hermitian structure is given. Denote by $\mu\colon E\to Lie(\mathcal G)\otimes\mathbb R^3$ the corresponding momentum map. Recall that with suitable identifications the Clifford multiplication in dimension four is the map $\mathbb H\otimes_{\mathbb R}\mathbb H\rightarrow\mathbb H,\ x\otimes y\mapsto \bar xy$. Hence, identifying $V^*$ with $\mathbb H$ we get a variant of the Clifford multiplication $Cl\colon V^*\otimes E\to E$. Let $b$ be a connection on the trivial $\mathcal G$-bundle, i.e. $b\in\Om^1(X;Lie (\mathcal G))$, and $u\in C^\infty(X; E)$ be a spinor. The following equations for $(b,u)$%
\begin{equation*}
       \dirac_b u=Cl(\nabla^b u)=0,\quad
       F_b^++\mu (u)=0,
\end{equation*}
are called Taubes-Pidstrygach equations. 

For a real parameter $\e$ consider the following perturbation 
\begin{equation}\label{Eq_ToyPertTP}
       \dirac_{b_\e} u_\e=Cl(\nabla^{b_\e} u_\e)=0,\quad
       \e F_{b_\e}^++\mu (u_\e)=0.
\end{equation}
Let $(u_0, b_0)$ be a solution for $\e=0$. Assuming that the hyperK\"ahler reduction $E\hkred\mathcal G=\mu^{-1}(0)/\mathcal G$ is smooth, a little thought shows that the map $v_0\colon \mathbb R^4\xrightarrow {\ u_0\ }\mu^{-1}(0)\rightarrow E\hkred\mathcal G$ satisfies the Fueter equation%
\begin{equation*}
 Cl(d_Xv_0)=\frac {\partial v_0}{\partial x_0}- I_1\frac {\partial v_0}{\partial x_1}- I_2\frac {\partial v_0}{\partial x_2}- I_3\frac {\partial v_0}{\partial x_3}=0. 
\end{equation*}
Here $I_j$ are complex structures on $E\hkred\mathcal G$ and $x_j$ are coordinates on $\mathbb R^4$.


Consider now the special case when $E$ consists of all $a\in \Om^1(W^+;\mathfrak g)$ satisfying a suitable asymptotic condition at infinity. Here $\Om^1(W^+;\mathfrak g)$ is equipped with the $L_2$-scalar product and the complex structures are induced from $W^+$. Let $\mathcal G^0$ denote the group of all gauge transformations based at infinity. Then the action of  $\mathcal G^0$ is compatible with the quaternion Hermitian structure and the moment map is $\mu (a)=F_a^+$. A straightforward computation shows that for our choice of $E$ the Taubes-Pidstrygach equations for a pair $(b,u)=(b,a)$ are exactly equations~(\ref{Eq_ToySpin7Instantons}). Moreover, the hyperK\"ahler reduction $E\hkred\mathcal G^0$ is the moduli space of \emph{framed} asd connections $\mathcal M_{asd}^0$, which is smooth. Then the perturbation of the form~\eqref{Eq_ToyPertTP} corresponds to a scaling of the metric on $W^+$ and leads to the Fueter maps from $\mathbb R^4$ to $\mathcal M_{asd}^0$.


\section{Preliminaries}\label{Sect_Preliminaries}

\subsection{Differential forms on fibre bundles}\label{Subsect_DiffFormsOnFibreBundles} 
Let $X$ be a manifold and $H$ be a Lie group. Let $M$ be another manifold endowed with an action of the group $H$. Pick a principal $H$-bundle $\pi:Q\rightarrow X$ equipped with a connection $\varphi$ and denote by $\bM \xrightarrow{\ \rho \ }X$ the associated fibre bundle: $\bM= Q\times_H M$.
The connection $\varphi$ determines a splitting $\mathrm T\bM=\mathcal H_{\ssst \bM}\oplus\mathcal V_{\ssst \bM}$ into horizontal and vertical subbundles and therefore the space of differential forms om $\mathbb M$ is bigraded with $\Om^{p,q}(\bM)=
\Gamma\duzhky{\Lambda^p\mathcal H_{\ssst\bM}^* \otimes \Lambda^q\mathcal V_{\ssst\bM}^* }$.


In the sequel, we will often make use of a construction called a ``change of fibre". We illustrate this with the following example.  The infinite dimensional graded vector space $\Om(M)$ inherits an action of $H$ and therefore we have the associated vector bundle $\mathcal E\rightarrow X$ of infinite rank: $\mathcal E=Q\times_H\Om(M)$. One can think of  $\mathcal E$ as the fibre bundle obtained by replacing each fibre $\mathbb M_x\cong M$ by $\Om(\mathbb M_x)$.

The connection $\varphi$ induces the covariant derivative $\nabla^\varphi\colon \Gamma(\mathcal E)\rightarrow \Om^1(\mathcal E)$, which extends to the map $d_\varphi\colon \Om^p(\mathcal E)\rightarrow\Om^{p+1}(\mathcal E)$. Identifying $\Om^p(\mathcal E^q)$ with $\Om^{p,q}(\mathbb M)$ we see that $d_\varphi\colon \Om(\mathbb M)\rightarrow \Om(\mathbb M)$ is a homomorphism of bidegree $(1,0)$. On the other hand, the exterior derivative $d\colon\Om(M)\rightarrow\Om(M)$ is $H$-invariant and therefore induces a homomorphism %
$d_v\colon \Om(\mathbb M)\rightarrow \Om(\mathbb M)$ of bidegree $(0,1)$. Unlike in the case of the Cartesian product, the exterior derivative on $\mathbb M$ has one more component, which we describe next. 

Let $K_\xi$ denote the \Kvf of the $H$-action on $M$ corresponding to $\xi\in\mathfrak h=Lie(H)$. The contraction 
$\mathfrak h\otimes\Om(M) \rightarrow \Om(M),\ \xi\otimes\om\mapsto \imath_{K_\xi}\om $
defines a homomorphism of vector bundles $ad\, Q\otimes\mathcal E\rightarrow\mathcal E$. Then the
curvature form $\Phi$ of the connection $\varphi$ induces the map $\imath_\Phi\colon \Om^p(\mathcal E^q)\rightarrow \Om^{p+2}(\mathcal E^{q-1})$ via a combination of wedging and contraction.

\begin{thm}[\cite{BismutLott:95}]\label{Thm_ComponentsOfDifferential}
The exterior derivative $d_{\ssst\bM}\colon\Om(\bM)\rightarrow\Om(\bM)$ decomposes as
follows: %
$  d_{\ssst\bM}=d_v+ d_{\varphi} -\imath_\Phi.$
\end{thm}

Let $P\rightarrow M$ be a principal $G$--bundle. We assume that a lift of the 
$H$-action to $P$ is provided such that the actions of $G$ and $H$ commute.
Then the associated bundle $\mathbb P =Q\times_H P$ yields a principal
$G$-bundle over $\mathbb M$:
\begin{displaymath}
 \begin{CD}
 Q\times P   @> > >  \mathbb P \\
 @VVV                @VVV\\
 Q\times M @> > >  \ \mathbb M.
 \end{CD}
\end{displaymath}

Denote by $\mathcal E^0(ad\, P)$ (respectively $\mathbb A$) the fibre bundle obtained by replacing each fibre $\mathbb M_x$ by $\Om^0(\mathbb M_x;\, ad\, \mathbb P)$ (respectively $\mathcal A(i^*_x\mathbb P)$), where $i_x\colon \mathbb M_x\hookrightarrow \mathbb M$ is the inclusion. The connection $\varphi$ determines the covariant derivatives on both $\mathcal E^0(ad\, P)$ and $\mathbb A$ as well as the inclusions $\hat\cdot\colon \Om^1 (\mathcal E^0(ad\, P))\hookrightarrow \Om^1(\mathbb M; ad\,\mathbb P)$ and $\hat\cdot\colon \Gamma(\mathbb A)\hookrightarrow\mathcal A(\mathbb P)$. The letter inclusion is best seen by thinking of connections as 1-forms on the corresponding principal bundles.  For any $a\in\Gamma(\mathbb A), b\in \Om^1(\mathcal E^0(ad\, P))$ the sum $A=\hat a+\hat b$ is a connection on $\mathbb P$. Vice versa, any connection $A$ on $\mathbb P$ can be decomposed as $\hat a+\hat b$ for some $a,b$ as above.

The proof of the following Proposition can be obtained, for instance, by a
straightforward application of Theorem~\ref{Thm_ComponentsOfDifferential} to the local
representations of connection forms. We omit the details. 

\begin{prop}\label{Prop_CompOfCurvature}
 For a connection $A=\hat a+\hat b$ the components of the curvature $ F_A\in\Om^2(\mathbb M; ad\,\mathbb P)$ are given by the following formulae:
 \begin{align}
    F_{\ssst A}^{0,2} &= F_a;\label{Eq_CurvatureComponent02}\\%
    F_{\ssst A}^{1,1} &= \nabla^\varphi a
       +\nabla^a b; \label{Eq_CurvatureComponent11}\\%
    F_{\ssst A}^{2,0} &= -\imath_{\ssst\Phi}a + d_\varphi b+
       [b,b].\label{Eq_CurvatureComponent20}
  \end{align}
\end{prop}

In formulae~\eqref{Eq_CurvatureComponent11},\eqref{Eq_CurvatureComponent20} the following notations are used.  First notice that each fibre of the bundle $\mathcal E^0(ad\, P)$ is naturally a Lie algebra. Then the term $[b,b]$ in~(\ref{Eq_CurvatureComponent20}) means a
combination of the Lie brackets and wedging. Further, for each fixed $x\in X$ the value of $a$ at $x$  gives a connection $\nabla^{a(x)}$ on $\Om^0(i_x^*ad\,\mathbb P)$. On the other hand, the value of $b$ at $x$ lies in $\Om^0(i_x^*ad\,\mathbb P)\otimes T_x^*X$ and therefore the
(vertical) covariant derivative $\nabla^{a(x)}b(x)$ is well defined. It is abbreviated as $\nabla^ab$ in~(\ref{Eq_CurvatureComponent11}).


\subsection{The group Spin(7), some subgroups and representations}\label{Subsect_GroupSpin7} 
Denote by $\mathbb H$ the $\mathbb R$-algebra of quaternions and by $Sp(1)$ the group of all quaternions of unit length. The basic
complex representation $W$ of $Sp(1)$ is given by%
\begin{equation}\label{Eq_BasicSp1Action}
  (q,x)\mapsto qx,\qquad q\in Sp(1),\  x\in\mathbb H\cong W.
\end{equation}

Consider the group $K=\bigl (Sp(1)\times Sp(1)\times Sp(1)\bigr )/\pm 1$, where $-1$ acts componentwise. It is
convenient to give a certain label to each component of $K$ as follows%
\begin{equation}\label{Eq_GpK}
  K=\bigl ( Sp_+\duzhky 1\times Sp_-(1)\times Sp_0(1)\bigr )/\pm 1.
\end{equation}
Define the action of  $K$  on $\mathbb R^8\cong \mathbb H\oplus\mathbb H$ by $[q_+, q_-, q_0]\cdot (x, y)=(q_+x\,\bar q_-,\; q_+ y\,\bar q_0)$. Denote by $U$ the corresponding representation, which is clearly the direct sum of two real irreducible $K$-representations $E$ and $F$ such that%
\begin{equation}\label{Eq_Spaces_EF}
 \c E\cong W^+\!\otimes W^-,\quad \c F\cong W^+\!\otimes W^0.
\end{equation}


Let $\theta$ (respectively $\eta$) denote the projection of $\mathbb R^8=\mathbb H\oplus\mathbb H$ onto the
first (resp. second) component. It is convenient to think of $\theta$ and $\eta$ as $\mathbb
H$-valued 1-forms on $\mathbb R^8$. The following 4-form %
\begin{equation}\label{Eq_Standard4Form}
 \Omega=-\dfrac 1{24}\mathrm{Re}\,\Bigl (
   \theta\wedge\bar\theta\wedge\theta\wedge\bar\theta -%
   6\,\theta\wedge\bar\theta\wedge \eta\wedge\bar\eta+
   \eta\wedge\bar\eta\wedge \eta\wedge\bar\eta \Bigr )
\end{equation}
is $K$-invariant. Hence  we obtain~\cite{BryantSalamon:89} $K\subset Stab_\Omega= Spin(7)\subset SO(8)$.


Think of $\mathbb R^8$  as a $Spin(7)$-representation via the inclusion $Spin(7)\subset SO(8)$. The linear map %
\begin{equation*}
  \TOmega: \Lambda^2 \bigl (\mathbb R^8\bigr )^*\rightarrow \Lambda^2\bigl (\mathbb R^8\bigr )^*, \qquad
  \omega\mapsto - *(\Omega\wedge\omega )
\end{equation*}
has two eigenvalues $3$ and $-1$. The corresponding eigenspaces $\Lambda^2_+\bigl (\mathbb R^8\bigr )^*$ and $\Lambda^2_-\bigl (\mathbb R^8\bigr )^*$ are irreducible $Spin(7)$-representations of dimension 7 and 21 respectively~\cite{Bryant:87}. %
One can check that the collection of 2-forms $(\omega_1,\dots , \omega_7)$, where %
\begin{equation}\label{Eq_BasisOfSelfdualForms}
 \begin{aligned}
   \om_1i+\om_2j+\om_3k       &= \theta\wedge\bar\theta - \eta\wedge\bar\eta,\\%
   \om_4+\om_5i+\om_6j+\om_7k &=\bar\theta\wedge\eta,
 \end{aligned}
\end{equation}
is a basis of $\Lambda_+^2\bigl (\mathbb R^8\bigr )^*$. Hence we have an isomorphism of $K$-rep\-re\-sen\-ta\-tions 
\begin{equation}\label{Eq_Lambda2PlusDecomposition}
  \Lambda^2_+U^*\cong \mathfrak{sp}_+(1)\oplus V, 
\end{equation}
where $V$ is the standard representation of $SO(4)=(Sp_-(1)\times Sp_0(1))/{\pm 1}$.

Taking into account~(\ref{Eq_Spaces_EF}) it is easy to see that there is essentially a unique homomorphism $E\otimes F\rightarrow V$. Its complexification is the four-dimensional Clifford multiplication $ Cl: W^+\!\otimes W^+\!\otimes W^- \rightarrow W^-$ twisted by $W^0$.

Denote by %
 $\Pi\colon \Lambda^2U^*\cong %
     \Lambda^2E^*\,\oplus E^*\otimes F^* \oplus\,\Lambda^2F^*\longrightarrow%
       \Lambda_+^2U^*%
$
 the natural projection. Combining the above observations we get that $\Pi$ maps $E^*\otimes F^*$ onto the $V$-component of $\Lambda_+^2U^*$ only and the composition
\begin{equation}\label{Eq_ProjPiPrime} 
\Pi'\colon E^*\otimes F^*\rightarrow \Lambda_+^2U^*\rightarrow V
\end{equation}
 is the (twisted) Clifford multiplication.

On the other hand, both $\Lambda^2_+E^*$ and $\Lambda^2_+F^*$ are naturally isomorphic to $\mathfrak{sp}_+(1)$, which in turn is the other irreducible component of $\Lambda_+^2U^*$. Then $\Pi$ maps $\Lambda^2E^*\oplus\Lambda^2F^*$ onto the $\mathfrak{sp}_+(1)$-component of $\Lambda_+^2U^*$ only and  a computation shows that the composition $\Pi''\colon \Lambda^2E^*\oplus\Lambda^2F^*\rightarrow \Lambda_+^2U^*\rightarrow \mathfrak{sp}_+(1)$ is given by%
\begin{equation*}
\Pi'' (\alpha,\beta)= \alpha^+-\beta^+,
\end{equation*}
where the above identifications are understood.

\begin{rem}\label{Rem_LimitOfProjections}
Although it is natural to take the standard Euclidean metric on
$\mathbb R^8$, one can also consider the following perturbation%
\begin{equation*}
   \begin{aligned}
        & g_\e=g_{\ssst {E}}\oplus \e g_{\ssst {F}}=\Re\bigl ( \theta\otimes\bar\theta + \e\, \eta\otimes\bar\eta  \bigr ), 
                     \qquad \\%
       &  \Om_\e=-\dfrac 1{24}\Re\Bigl (  \theta\wedge\bar\theta\wedge \theta\wedge\bar\theta - %
     6\e\,\theta\wedge\bar\theta\wedge \eta\wedge\bar\eta+ %
     \e^{2}\,\eta\wedge\bar\eta\wedge \eta\wedge\bar\eta \Bigr )
    \end{aligned}
\end{equation*}
for $\e>0$. Then by tracing the above computations we get
\begin{equation*}
 \Pi_\e=\Pi_\e'+\Pi_\e'',\qquad \Pi_\e'=\e\Pi',\quad \Pi_\e''(\alpha,\beta)=\alpha^+-\e^{-1}\beta^+.
\end{equation*}
\end{rem}


\subsection{Spin(7)-structures on spinor bundles} For an oriented Riemannian manifold $\mathbb W^8$ a $Spin(7)$-structure is a principal $Spin(7)$-subbundle of the $SO(8)$-bundle of orthonormal oriented frames or, equivalently, a $4$-form $\Omega$, whose restriction to each tangent space lies in the $SO(8)$-orbit of the standard 4-form~(\ref{Eq_Standard4Form}). In this section, following~\cite{BryantSalamon:89}, we describe  a $Spin(7)$-structure on the total space of the spinor bundle over a four-manifold.

From now on $X$ denotes a smooth closed oriented Riemannian manifold. Let %
$\tilde Q\xrightarrow{\ \tilde\pi\ } X$ be the $SO(4)$-principal bundle of oriented isometries %
$\varkappa: T_xX\rightarrow\mathbb H$. Denote by $\tilde\theta\in\Omega^1(\tilde Q;\mathbb H)$ the
tautological 1-form, i.e. $\tilde\theta (\mathrm v)=\tilde f(\tilde\pi_*\mathrm v)$, where %
$\mathrm v\in T_{\tilde f}\tilde Q$. We also assume that $X$ is spin and pick a %
$Spin(4)=Sp_+(1)\times Sp_-(1)$-structure $Q\xrightarrow{\ \pi\ }X$, which is a double cover of
$\tilde Q$. The Levi-Civita connection $(\varphi, \psi)$ is an equivariant 1-form on $Q$ with values in
$\mathfrak{sp}_+(1)\oplus\mathfrak{sp}_-(1)\cong\ImH\oplus\ImH$. Let $x$ be the quaternionic variable on $\mathbb H$. Denote $\eta=dx-\varphi x\in\Om^1(Q\times\mathbb H;\,\mathbb H)$ and put
\begin{equation*}
  \Omega=-\dfrac 1{24}\Re\Bigl (  \theta\wedge\bar\theta\wedge\theta\wedge\bar\theta -%
    6\, \theta\wedge\bar\theta\wedge\eta\wedge\bar\eta +%
    \eta\wedge\bar\eta\wedge\eta\wedge\bar\eta \Bigr )
    \in\Omega^4(Q\times\mathbb H),
\end{equation*}
where $\theta$ is the pull-back of $\tilde\theta$.

Further, define an action of $Spin(4)$ on $Q\times\mathbb H$ by the rule $(f, x)\cdot (q_+, q_-)=(f\cdot (q_+, q_-),\; \bar q_+x)$. Clearly, the quotient space is the positive spinor bundle $\rho\colon\mathbb W^+\rightarrow X$. It is easy to check that the 4-form $\Om$ is $Spin(4)$-invariant and basic. Therefore $\Om$ descends to a 4-form (denoted by the same letter) on the total space of  $\mathbb W^+$ and defines a $Spin(7)$-structure. The corresponding metric is given by $ g=\Re \bigl ( \theta\otimes\bar\theta + \eta\otimes\bar\eta \bigr )$.

\begin{rem}
 One can replace the $Spin(4)$-bundle $Q$ in the above setting by a principal $K$-bundle $Q'$ such that $Q'/Sp_0(1)=\tilde Q$. In particular, we can choose an embedding $S^1\hookrightarrow Sp_0(1)$ and take $Q'$ as a principal $Spin^c(4)=\bigl ( Sp_+(1)\times Sp_-(1)\times S^1 \bigr )/\pm 1$ bundle. Therefore we could have started with an arbitrary closed smooth oriented Riemannian four-manifold, since for such manifolds a $Spin^c(4)$-structure always exists. However, in this case one needs to make a choice of connection on the determinant bundle $Q'/Spin(4)$. 

On the other hand, choosing the $Spin(4)$-action differently one can obtain $Spin(7)$-structures on $T^*X$ or $\underline{\mathbb R}\oplus \Lambda^2_+ T^*X$. In these cases the existence of $Spin(4)$-structures is also not needed.
\end{rem}



\section{Spin(7)-instantons and Taubes-Pid\-stry\-gach equations }\label{Sect_ASDOnEightMnflds}


\subsection{\TP equations: a perturbation}\label{Subsec_PerturbedTPeqns}

 In this section we study formal aspects of a perturbation of the \TP equations.  We first sketch the \TP construction in a form suitable for our purposes. The construction involves two manifolds (\emph{a source} and \emph{a target}).
\paragraph{Source manifold.}  Let $X$ be a four-dimensional oriented Riemannian
manifold, which is referred to as a source manifold in the sequel. For the sake of simplicity we assume as
before that $X$ is spin and denote by $Q_+=Q/Sp_-(1)$ the principal bundle of $\mathbb W^+$.

Let $\mathcal G$ be a Lie group whose Lie algebra $Lie(\mathcal G)=\mathcal L$ is endowed with an
$Ad$--invariant scalar product. Assume a homomorphism %
$\alpha\colon Sp(1)\rightarrow Aut(\mathcal G)$ is given and put $\hat{\mathcal
G}=\mathcal G\rtimes Sp(1)$. Let $\hat Q$ be $\mathcal G\times Q_+$ as fibered space but considered as a principal $\hat{\mathcal G}$-bundle.


Further, the decomposition of the vector space $Lie(\hat{\mathcal G})=\hat{\mathcal L}=\mathcal L+\mathfrak{sp}(1)$ is invariant \wrt $Sp(1)\hookrightarrow\hat{\mathcal G}$. Since we have an inclusion $i\colon Q_+\hookrightarrow\hat Q$, for any connection 
$B\in\Om^1(\hat Q;\;\hat{\mathcal L})$ we obtain
\begin{equation}\label{Eq_DecompOfConnection}
   i^*B=b+\varphi,
\end{equation}
where $\varphi$ is a connection on $Q_+$ and $b$ is a 1-form on $X$ with values in $\BundleLieAlgG=Q_+\times_{Sp_+(1)}\mathcal L$. A simple computation shows that $F_B=F_b+F_\varphi$, where $F_\varphi$ is the curvature of $\varphi$ and  $F_b=d_\varphi b+ [b,b]\in\Om^2(X;\, \BundleLieAlgG)$. %
In what follows we consider only those connections $B$, whose $\mathfrak{sp}(1)$-component $\varphi$ is the Levi-Civita connection on $Q_+$.

\paragraph{Target manifold.} The other ingredient of the construction is a hyper\-K\" ah\-ler manifold $M$ called the target space. Recall that a Riemannian manifold $(M, g)$ is called \hK if $g$ is \kahler \wrt three complex structures $(I_1, I_2, I_3)$ satisfying the quaternionic relations. Denote by $(\om_1, \om_2, \om_3)$ the corresponding \kahler 2-forms and put $ \om = \om_1 i+\om_2 j+\om_3 k\in\Om^2(M;\,\ImH)$. We also assume that $\hat{\mathcal G}=\mathcal G\rtimes Sp(1)$ acts isometrically on $M$ such
that the following two conditions hold:
\begin{itemize}
	\item[(A)] the action of $\mathcal G\subset\hat{\mathcal G}$ is tri-Hamiltonian (in
particular preserves complex structures);
	\item[(B)] the action of $Sp(1)\subset\hat{\mathcal G}$ is \emph{permuting}, i.e. for all  $q\in Sp(1)$ we have
             $\left (L_q\right )^*\om=q\om\bar q$, where $L_q(m)= qm$.
\end{itemize}
\paragraph{Nonlinear Dirac operator.} Below we sketch a construction of a Dirac operator acting on sections of a nonlinear bundle $\rho\colon\mathbb M\rightarrow X$. We refer to~\cite{Haydys_ahol:08} for more details. 

Put $\mathbb M = \hat Q\times_{\hat{\mathcal G}}M=Q_+\times_{Sp_+(1)}M$. For each point $(\hat f, m)\in \hat Q\times M$ we have the following exact sequence%
\begin{equation*}
   0\longrightarrow \hat{\mathcal L}\longrightarrow T_{\hat f}\hat Q\times T_mM
	\xrightarrow{\ \;\tau_*\ }T_{\tau(\hat f,m)}\mathbb M \longrightarrow 0,
\end{equation*}
where $\tau\colon\hat Q\times M\rightarrow \mathbb M$ is the natural projection map. A connection
$B$ on $\hat Q$ can be regarded as a $\hat{\mathcal G}$-invariant splitting %
$T\hat Q=\hat{\mathcal H}\oplus\hat{\mathcal V}$ and induces a similar splitting $\mathrm{T}\mathbb M=\mathcal H\oplus\mathcal V$,
where $\mathcal V=\tau_*(TM)=\ker\rho_*,\ \mathcal H= \tau_*(\hat{\mathcal H})=\rho^*TX$. 

Further, one can think of $f\in Q_+$ as a quaternionic structure on
$T_{\pi_+(f)}X\cong\hat{\mathcal H}_{(g,f)}\subset T_{(g,f)}\hat Q$ and therefore the horizontal
bundle $\hat{\mathcal H}$ is equipped with a quaternionic structure
$(J_1, J_2, J_3)$. For the subbundle $\hat E^-= Hom_{\mathbb H}(\hat{\mathcal H}, TM)$ of %
$Hom_{\mathbb R} (\hat{\mathcal H}, TM) \rightarrow \hat Q\times
M$ we have the natural complement
\begin{equation*}
   \hat E^+=\set{A\in\hom R{\mathcal H}{TM}\left |\right.\; I_1AJ_1+I_2AJ_2+I_3AJ_3=A}.
\end{equation*}
It follows from assumptions (A) and (B) that the splitting %
$\hat{\mathcal H}^*\otimes TM=\hat E^-\oplus \hat E^+$ is $\hat{\mathcal G}$-invariant and therefore
we obtain 
\begin{equation}
   \mathcal H^*\otimes \mathcal V= E^-\oplus E^+,
\end{equation}
where $E^\pm\rightarrow\mathbb M$ is the factor of $\hat E^\pm\rightarrow \hat Q\times M$ by the
$\hat{\mathcal G}$-action. We denote by $\mathcal C\colon \mathcal H^*\otimes \mathcal V\rightarrow E^-$ the projection onto the first subbundle.

\begin{rem}
 In the case $M=\mathbb H,\ \mathcal G=\{ 1\}$ we have $\mathbb M=\mathbb W^+,\ E^-=\rho^*\mathbb W^-$ and %
$\mathcal C\colon \rho^*T^*X\otimes\rho^*\mathbb W^+\rightarrow\rho^*\mathbb W^-$ is the usual Clifford multiplication.
\end{rem}

Let $ev\colon\Gamma(\mathbb M)\times X\rightarrow\mathbb M$ be the evaluation map. Think of sections of $ev^*E^-$ as maps associating to each $u\in\Gamma(\mathbb M)$ a section of $ev^*E^-|_X\cong u^*E^-$.

\begin{defn}\label{Defn_GenerDiracOp}
   The following section of $ev^*E^-$
\begin{equation*}
   \dirac_b\colon \Gamma(\mathbb M)\xrightarrow{\ \nabla^B\ }%
	\Gamma(T^*X\otimes u^*\mathcal V)\cong
	\Gamma\duzhky{u^*\duzhky{  \mathcal H^*\otimes \mathcal V }}\xrightarrow{\ \mathcal C\ }%
	\Gamma(u^*E^-)
\end{equation*}
is called a (generalized) \emph{Dirac operator}, where  $B$ is given
by~(\ref{Eq_DecompOfConnection}).
\end{defn}

A particular role in the sequel is played by  Dirac operators $\dirac_0$ corresponding to $\mathcal G=\set{1}$ (and hence $b=0$). Examples of generalized Dirac operators and corresponding harmonic spinors (i.e. solutions of the equation $\dirac_b u=0$) can be found in~\cite{Haydys_ahol:08}. In section~\ref{Subsect_ASDasPSW} we present an example of a Dirac operator with an infinite-dimensional target space.


Below it will be useful to rewrite the equation $\dirac_b u=0$ in the equivariant setup as
follows. Recall that a section $u$ of $\mathbb M\rightarrow X$ can be identified with an
equivariant map $\hat u\colon \hat Q\rightarrow M$. Then the covariant derivative $\nabla^B u$
is given by the restriction $\hat u_*^h$ of the differential $\hat u_*$ to the horizontal subspace.
It is clear from the above description that $u$ is harmonic iff
\begin{equation}\label{Eq_AholomSection}
    I_1\hat u_*^hJ_1+I_2\hat u_*^hJ_2+I_3\hat u_*^hJ_3=\hat u_*^h,
\end{equation}
or, equivalently, iff pointwise $\hat u_*^h$ has no $\mathbb H$--linear component.

\paragraph{\TP equations.} Recall that the action of $\mathcal G\subset\hat{\mathcal G}$ on $M$ is tri-\-Ham\-il\-tonian. Let $\mu\colon M\rightarrow \ImH\otimes\mathcal L\cong\mathfrak{sp}(1)\otimes\mathcal L$ be a momentum map. We also assume that $\mu$ is $Sp(1)$-equivariant, where $\mathcal L$ is considered as the $Sp(1)$-representation via the restriction of the adjoint representation of $\hat{\mathcal G}$ to $Sp(1)\subset \hat{\mathcal G}$. Then the map $id\times\mu: Q_+\times M\rightarrow Q_+\times\duzhky{\mathfrak{sp}(1)\otimes\mathcal L}$  can be
identified with a section $\nu\in\Gamma\duzhky{\mathbb M;\,\rho^*\duzhky{\Lambda^2_+X\otimes\BundleLieAlgG}}$.
Finally, to any spinor $u\in\Gamma(\mathbb M)$ we associate a self-dual 2-form $\nu\comp u\in\Om^2_+(X;\, \BundleLieAlgG)$.

\begin{defn}
 The following system of first order partial differential equations
 \begin{equation}\label{Eq_PSW}
     \dirac_b u=0,\quad
     F_b^+ + \nu\comp u=0,
 \end{equation} 
for a pair $(u,b)\in\Gamma(\mathbb M)\times\Om^1(X;\,\BundleLieAlgG)$ is called \TP equations.
\end{defn}

Denote $\mathbb G=P_+\times_{Sp_+(1)}\mathcal G$. The gauge group $\Gamma(\mathbb G)$ acts on the configuration space $\Gamma(\mathbb M)\times\Om^1(X;\BundleLieAlgG)$
\begin{equation}\label{Eq_GaugeGpAction}
   g\cdot (u,b)= \bigl (g\cdot u,\; Ad_g\, b - (\nabla^\varphi g)g^{-1}\bigr )
\end{equation}
and preserves the space of solutions of~(\ref{Eq_PSW}). Denote by $\mathcal M_{TP}$ the corresponding moduli space.

\paragraph{A perturbation.} For a positive parameter $\e$ consider the following perturbation
\begin{equation}\label{Eq_PertPSW}
   \dirac_{b_\e} u_\e=0,\quad
   \e F_{b_\e}^+ + \nu\comp u_\e=0.
\end{equation}
Putting formally $\e=0$ we obtain the system
\begin{equation}\label{Eq_LimPSW}
   \dirac_{b_0} u_0 =0,\quad
   \ \nu\comp u_0 =0.
\end{equation}
Let $\mathcal M_{TP}^0$ denote the moduli space of solutions to the above system.

\begin{rem}
A perturbation analogous to~(\ref{Eq_PertPSW}) in the case of symplectic vortex equations have been studied in~\cite{SalamonEtAl:00,SalamonEtAl:02,GaioSalamon:05}.  For the classical Seiberg-Witten equations perturbation similar to~(\ref{Eq_PertPSW})  was studied in~\cite{Taubes:07_SW_Weinstein,Taubes:96_SW_Gr}. It is also interesting to observe that putting formally $\e=+\infty$ we obtain the system 
\begin{equation*}
   \dirac_{b_\infty} u_\infty=0,\quad
     F_{b_\infty}^+=0,
\end{equation*}
which was studied in~\cite{PidstrygachTyurin:92} in the case of the linear Dirac operator.
\end{rem}

From now on we assume that $0\in\mathcal L$ is a regular value of the momentum map $\mu$ and that
the group $\mathcal G$ acts freely on $\mu^{-1}(0)$. This assumption is crucial for the
following Proposition.  Denote by
$M_0=\mu^{-1}(0)/\mathcal G$ the \hK reduction of $M$. Then the permuting action of $Sp(1)$ on $M$
induces a permuting action on $M_0$ and we denote by $\mathbb M_0\rightarrow X$ the associated bundle
$Q_+\times_{Sp_+(1)} M_0$. Observe also that the projection $\mu^{-1}(0)\rightarrow M_0$
gives rise to the fibrewise map $\nu^{-1}(0)\rightarrow\mathbb M_0$. 

\begin{prop}\label{Prop_HarmSpinorHKred}
Assume $0\in\mathcal L\otimes\ImH$ is a regular value of the momentum map $\mu$ and 
$\mathcal G$ acts freely on $\mu^{-1}(0)$. Pick a spinor $u\in\Gamma(\mathbb M)$ such that $\nu\comp
u=0$ and denote by $v\in\Gamma(\mathbb M_0)$ its
projection. Then $\dirac_0 v=0$ if and only if there exists $b\in\Om^1(X;\,\BundleLieAlgG)$
such that $\dirac_b u=0$.
\end{prop}

\begin{proof}
  Let $\hat u$ and $\hat v$ be equivariant maps representing $u$ and $v$ respectively such that the
following diagram%
\begin{equation}\label{CD_LiftSpinors}
 \xymatrix{
  Q_+ \ar[rr]^{\hat u} \ar[drr]^{\hat v} & & \mu^{-1}(0)\ar[d]\\
                    & &   M_0
         }
\end{equation}
commutes. Pick a point $f\in Q_+$  and denote $m=\hat u(f)\in\mu^{-1}(0)\subset M$. Let $\mathcal
K_m\cong\mathcal L$ be the vector space spanned by Killing vectors at the point
$m$. Define the subspace $\EuScript H_m\subset T_mM$ by the following orthogonal
decomposition%
\begin{equation*}
 T_mM=\EuScript H_m\oplus \mathcal K_m\oplus I_1\mathcal K_m \oplus I_2\mathcal K_m \oplus
I_3\mathcal K_m.
\end{equation*}
Notice that $T_m\mu^{-1}(0)=\EuScript H_m\oplus \mathcal K_m$ and $T_{[m]}M_0$ can be
identified with $\EuScript H_m$.

The image of $\hat u_*: T_fQ_+\rightarrow T_m M$ is contained in $\EuScript
H_m\oplus \mathcal K_m$ and the projection to $\EuScript H_m$ yields the differential of
$\hat v$. Since $\mathcal G$ acts freely on $\mu^{-1}(0)$, for each $\mathrm v\in T_fQ_+$
there exists a unique $b(\mathrm v)\in\mathcal L$ such
that %
\begin{equation}\label{Eq_DefinitionOf_b-Form}
 \hat u_*(\mathrm v)-\hat v_*(\mathrm v)=-K_{b(\mathrm v)}(m).
\end{equation}
Then $b\in \Om^1(Q_+; \mathcal L)$ is basic. Further, for $q\in Sp_+(1)$ denote $R_q\colon Q_+\rightarrow Q_+,\ R_q(f)=f\cdot q$.  We have%
\begin{equation*}
 K_{b((R_q)_*\mathrm v)}(\bar q\, m)=(L_{\bar q})_*K_{b(\mathrm v)}(m)=
K_{Ad_{\bar q}\,b(\mathrm v)}(\bar q\, m),
\end{equation*}
where the first equality follows from~(\ref{Eq_DefinitionOf_b-Form}) and the $Sp(1)$-equivariance of both $\hat u$ and $\hat v$. Since the action of $\mathcal G$ is free we get %
 $\bigl ( R_q \bigr )^* b = Ad_{\bar q}\, b$, 
i.e. $b$ descends to a 1-form on $X$ with values in $\BundleLieAlgG$.

Let $B$ be the connection on $\hat Q$ determined by the Levi-Civita connection and the 1-form $b$
as in~(\ref{Eq_DecompOfConnection}). Then the covariant derivative of
$\hat u$ \wrt $B$ can be identified with the restriction of $\hat u_* +K_{b(\cdot )}(\hat u)$ to
the horizontal bundle $\mathcal H^+\rightarrow Q_+$ of the Levi-Civita connection. It remains to
note that by virtue of equation~(\ref{Eq_DefinitionOf_b-Form})  $\dirac_0v=0$  iff for
each point $f\in Q_+$ the restriction of the $\mathbb R$-linear map %
$\hat u_* +K_{b(\cdot )}(\hat u)$ to $\mathcal H_f^+$ has no $\mathbb H$-linear component, i.e.
$\dirac_b u=0$.
\end{proof}


Consider the following space%
\begin{equation*}
 \Gamma_0(\mathbb M_0)=\set{ v\in\Gamma(\mathbb M_0)\, |\ \;\text{there exists } \hat u \text{ s.t. diagram  \eqref{CD_LiftSpinors} commutes}}
\end{equation*}
and denote by $\mathcal H_0(\mathbb M_0)\subset \Gamma_0(\mathbb M_0)$ the subspace of harmonic spinors.

\begin{thm}\label{Thm_LimitingPSW_HarmSpinors}
Assume $0\in\mathcal L\otimes\ImH$ is a regular value of the momentum map $\mu$ and 
$\mathcal G$ acts freely on $\mu^{-1}(0)$. Then there exists a bijective correspondence between the moduli space $\mathcal M^0_{TP}$ of
solutions to limiting problem~(\ref{Eq_LimPSW}) and the subspace %
$\mathcal H_0(\mathbb M_0)$ of harmonic spinors.
\end{thm}

\begin{proof}
 It follows from Proposition~\ref{Prop_HarmSpinorHKred} that we have a map  from the space of
solutions of~(\ref{Eq_LimPSW}) to $\mathcal H_0(\mathbb M_0)$, which factors 
through $\mathcal M_{TP}^0$. To construct the inverse map pick a harmonic spinor
$v\in\Gamma_0(\mathbb M_0)$. Then $u,u'\in\Gamma(\mathbb M)$ satisfying $\nu\comp u=0=\nu\comp u'$ 
are lifts of $v$ iff there exists $g\in\Gamma(\mathbb G)$ such that $u'=g\cdot u$. Then it is
easy to check that for the corresponding 1-forms $b$ and $b'$ as in 
Proposition~\ref{Prop_HarmSpinorHKred} we have %
 $b'=Ad_g\, b - (\nabla^\varphi g)g^{-1}$
and the statement follows.
\end{proof}

Proposition~\ref{Prop_HarmSpinorHKred} and   Theorem~\ref{Thm_LimitingPSW_HarmSpinors} are analogues of Lemmata~4.5.7 and~4.5.9 in~\cite{Haydys:06_Thesis} and are also proved in a similar manner. Theorem~\ref{Thm_LimitingPSW_HarmSpinors} was independently discovered by Pidstrygach~\cite{Pidstrygach:06_Talk}.

Theorem~\ref{Thm_LimitingPSW_HarmSpinors} can be regarded as a quaternionic version of the relation between moduli spaces of solutions to the symplectic vortex equations and pseudoholomorphic curves in symplectic reductions~\cite{SalamonEtAl:00,SalamonEtAl:02,GaioSalamon:05}. More precisely, a four-manifold takes the place of a Riemann surface, the role of the symplectic vortex equations is played by the \TP equations, the symplectic reduction is replaced by the \hK reduction, and pseudoholomorphic curves become generalized harmonic spinors. However there is an important distinction between the complex and quaternionic cases. Whereas in the complex case the focus is on the target manifold  (or rather on its symplectic reduction), in the quaternionic case it is interesting to study both the target and the source with the help of the \TP equations. Indeed, even the choice of the simplest admissible target manifold $\mathbb H$ leads to the standard Seiberg-Witten theory, which carries information about the smooth structure of the source manifold.


\subsection{Spin(7)-instantons on spinor bundles as Taubes-Pid\-stry\-gach system}\label{Subsect_ASDasPSW} 

The main aim of this section is to prove that the $Spin(7)$-instanton equations on the total space of a spinor bundle
$\mathbb W^+\rightarrow X$ is an example of the \TP system with an infinite dimensional target space. Let us introduce some notation first.


Let $P\rightarrow\mathbb R^4$ be a principal $G$-bundle equipped with a framing at infinity. We assume that the group $Sp(1)$ acts on the total space of $P$ commuting with $G$ and descending to basic action~(\ref{Eq_BasicSp1Action}) on $\mathbb R^4$. We also assume that the $Sp(1)$-action is  compatible with the framing at infinity. Let $\mathcal A^0(P)$ and $\mathcal G^0(P)$ consist of connections and gauge transformations on $P$ respectively with a suitable asymptotic behaviour at infinity (see~\cite{Itoh:89} for details). One can think of $\mathcal A^0(P)$ and $\mathcal G^0(P)$ as the space of connections and the \emph{based} gauge group on $S^4$ respectively. 

Put $M=W^+$ in the set-up of Section~\ref{Subsect_DiffFormsOnFibreBundles} and consider the bundle $\mathbb P=Q_+\times_{Sp_+(1)}P\rightarrow\mathbb W^+$. For the principal bundle $Q_+$ denote by $\mathbb{A}, \mathbb G, \BundleLieAlgG$ the associated fibre bundles over $X$ with fibres %
$\mathcal A^0(P), \mathcal G^0(P), \mathcal L^0(P)=Lie(\mathcal G^0(P))$ respectively. One can think of  $\mathbb{A}, \mathbb G$ and $\BundleLieAlgG$ as the bundles obtained by replacing each fibre $\mathbb W^+_x$ by   %
$\mathcal A^0(i^*_x\mathbb P),  \mathcal G^0(i^*_x\mathbb P)$ and  $ad\, (i^*_x\mathbb P)$ respectively.

\begin{ex}[The Dirac operator for the target $\mathcal A^0(P)$]\label{Ex_DiracOpOn1Forms}
 The construction of the Dirac operator in the case of a flat target space is somewhat simpler and we can give a more direct description.

We first choose $\Om^1(\mathbb R^4)$ as the target space. The basic action~(\ref{Eq_BasicSp1Action}) induces the permuting action of $Sp_+(1)$.  Then $\Om^1(W^+)$ is isomorphic to $C^{\infty}(W^+)\otimes W^+$ as an $Sp_+(1)$-representation and we get a variant of the Clifford multiplication
\begin{equation*}
 Cl\colon W^+\otimes W^-\otimes \Om^1(W^+)\cong 
W^+\otimes W^-\otimes C^{\infty}(W^+)\otimes W^+
\longrightarrow C^{\infty}(W^+)\otimes W^-,
\end{equation*}
which differs from the standard Clifford multiplication just by tensoring with $C^\infty(W^+)$.
With these choices the space of spinors is $\Gamma \bigl (\mathcal E^1\duzhky{\mathbb W^+}\bigr )$ and the corresponding (untwisted) Dirac operator is given by the sequence%
\begin{equation*}
 \Gamma \bigl (\mathcal E^1\duzhky{\mathbb W^+}\bigr ) \xrightarrow{\ \nabla^\varphi\ }%
 \Gamma \bigl (T^*X\otimes \mathcal E^1\duzhky{\mathbb W^+}\bigr )
          \xrightarrow{\ Cl\ } \Gamma \bigl (\mathbb W^-\!\otimes\mathcal E^0(\mathbb W^+)\bigr )\cong \Gamma \bigl (\rho^*\mathbb W^-\bigr ).
\end{equation*}
It follows from~(\ref{Eq_Lambda2PlusDecomposition}) that $\Lambda^2_+\mathrm T^*\mathbb W^+\cong%
\rho^*\Lambda^2_+T^*X\oplus\rho^*\mathbb W^-$. Then using the identification $\Om^p(\mathcal E^q)\cong \Om^{p,q}(\mathbb W^+)$ as in Section~\ref{Subsect_DiffFormsOnFibreBundles} and recalling~\eqref{Eq_ProjPiPrime}, the above sequence yields:%
\begin{equation*}
 \Om^{0,1}(\mathbb W^+)\xrightarrow{\ d^{1,0}\, }%
 \Om^{1,1}(\mathbb W^+)\longrightarrow \Om^2_+(\mathbb W^+),
\end{equation*}
where the last map is the natural projection.

Now let us take $\mathcal A^0(P)$  instead of $\Om^1(\mathbb R^4)$ as the target space. The gauge group $\mathcal G=\mathcal G^0(P)$ acts on $\mathcal A^0(P)$ preserving the \hK structure. If we define the homomorphism $\alpha\colon Sp(1)\rightarrow Aut(\mathcal G)$ by
\begin{equation*}
 (qg)(p)=g(\bar q\,p),\qquad p\in P,\quad g\in\mathcal G,
\end{equation*}
then $\mathcal G^0(P)\rtimes Sp(1)$ acts on $\mathcal A^0(P)$ such that conditions
(A) and (B) are satisfied. 
The corresponding Dirac operator $\dirac_b$ is given by
\begin{equation*}
  \Gamma(\mathbb A) \xrightarrow{\;\nabla^B\;}\Om^1(X; \mathcal E^1(ad\, P))\cong%
   \Om^{1,1}(\mathbb W^+\!;\, ad\,\mathbb P)\longrightarrow%
   \Om^2_+(\mathbb W^+\!;\, ad\,\mathbb P).
\end{equation*}
Further, for any $a\in\Gamma(\mathbb A)$ we have $\nabla^B a= \nabla^\varphi a+\nabla^a b$. Thus, in the notations of Proposition~\ref{Prop_CompOfCurvature} formula~(\ref{Eq_CurvatureComponent11}) yields:%
\begin{equation*}
   \dirac_b(a)=(\nabla^\varphi a +\nabla^a b)^+= (F_{A}^{1,1})^+, \qquad\text{where } A=\hat a+\hat b.
\end{equation*}
Here $(F_A^{1,1})^+\in\Om^2_+(\mathbb W^+; ad\,\mathbb P)$ denotes the self-dual component of $F_A^{1,1}$. 
\end{ex}

Recall that the moment map of the $\mathcal G=\mathcal G^0(P)$-action on $M=\mathcal A^0(P)$ is given by %
   $\mu (a)=F_a^+\in\Om^2_+(\mathbb R^4;\, ad\, P)\cong\ImH\otimes\mathcal L$. Thus we get the corresponding section $\nu$.

\begin{thm}\label{Thm_ASDarePSW}
   A connection $A=\hat a+\hat b$ on $\mathbb P\rightarrow \mathbb W^+$ is anti-self-dual iff the
pair $(a,b)\in\Gamma(\mathbb A)\times\Om^1(X;\,\BundleLieAlgG)$ is a solution to the
Taubes-Pidstrygach-type equations with the target manifold $M=\mathcal A^0(P)$%
   \begin{numcases}{}
      \dirac_b a=0,\label{Eq_PSW_FirstEq}\\
      F_b^+ + \nu\comp a=-(\imath_\Phi a)^+,\label{Eq_PSW_SecEq}	
   \end{numcases}
where $\Phi\in\Om^2(X;\,\Lambda^2_+X)$ is the curvature form of the component $\varphi$ of the Levi-Civita connection.
\end{thm}

\begin{proof} It follows from the results of Section~\ref{Subsect_GroupSpin7} that a connection $A$ is asd iff %
$(F_A^{1,1})^+=0$ and $(F_A^{2,0}+F_A^{0,2})^+=0$. Recalling formulae~(\ref{Eq_CurvatureComponent02}) and~(\ref{Eq_CurvatureComponent20}) one immediately sees that the second equation is equivalent to~(\ref{Eq_PSW_SecEq}). On the other hand, we have shown in Example~\ref{Ex_DiracOpOn1Forms} that the first equation is equivalent to~(\ref{Eq_PSW_FirstEq}). \end{proof}

Notice also that the asd equations with respect to the form $\Om_\e$ (see Remark~\ref{Rem_LimitOfProjections}) lead to the following perturbation of equations~(\ref{Eq_PSW_FirstEq}),(\ref{Eq_PSW_SecEq}): 
\begin{equation*}
 \begin{cases}
   \dirac_{b_\e} a_\e=0,\\
      \e\duzhky{F_{b_\e}+\imath_\Phi\, a_\e}^+ + \nu\comp a_\e=0.
 \end{cases}
\end{equation*}
Then, the formal limiting form of the above equations is%
\begin{equation}\label{Eq_LimitingASDInstantons}
  \begin{cases}
    \dirac_{b_0} a_0=0,\\
     F_{a_0}^+=0,
   \end{cases}
\Longleftrightarrow\quad
 \begin{cases}
      (F_{A_0}^{1,1})^+=0,\\
      (F_{A_0}^{0,2})^+=0,
  \end{cases}
\qquad A_0=\hat a_0+\hat b_0.
\end{equation}

On the other hand, the group $\mathcal G^0(P)$ of gauge transformations based at infinity acts freely on the space of asd-connections and the quotient is the space of framed asd connections $\mathcal M_{asd}^0$. In other words, $\mathcal M_{asd}^0$ is the \hK reduction of $\mathcal A^0(P)$  \wrt $\mathcal G^0(P)$. We denote $\mathbb M_{asd}^0=Q_+\times_{Sp_+(1)}\mathcal M_{asd}^0$. One can think of $\mathbb M_{asd}^0\rightarrow X$ as the fibre bundle obtained by replacing each fibre $\mathbb W_x^+\cong\mathbb R^4$ of the usual spinor bundle $\mathbb W^+\rightarrow X$ by the moduli space of framed instantons over $\mathbb W_x^+$. Then from Theorem~\ref{Thm_LimitingPSW_HarmSpinors} we obtain the following corollary.


\begin{cor}\label{Cor_InstantonsAndHarmSpinors}
There exits a natural bijective correspondence between the moduli space of solutions to equations~(\ref{Eq_LimitingASDInstantons})  and the subspace  $\mathcal H_0(\mathbb M_{asd}^0)$ of harmonic spinors.\qed
\end{cor}

\begin{rem}
When this paper has been essentially ready for publication, S. Donaldson communicated to the author a direct proof of Corollary~\ref{Cor_InstantonsAndHarmSpinors}, which has appeared in~\cite{DonaldsonSegal:09}. The reader can also find there, among other things, the relevance of  Corollary~\ref{Cor_InstantonsAndHarmSpinors} to the compactification of the moduli space of higher dimensional instantons.
\end{rem}

\begin{ex}
 Let $X^4$ be a compact \hK manifold, so that $Q_+$ is flat.  Then~\cite{Haydys_ahol:08} harmonic spinors are exactly aholomorphic maps $X\rightarrow \mathcal M_{asd}^0$. It follows from~\cite[Cor.\,2]{Haydys_ahol:08}  that each aholomorphic map $X\rightarrow \mathcal M_{asd}^0$ is constant. Hence we have $\mathcal H(\mathbb M_{asd}^0)=\mathcal M_{asd}^0$.
\end{ex}


\section{Spin(7)-instantons and Cayley fibrations}\label{Sect_InstantonsAndCayleyFibr}

In this section we show that a suitable modification of Corollary~\ref{Cor_InstantonsAndHarmSpinors} holds for $Spin(7)$-manifolds equipped with a structure of Cayley fibration. It is convenient to fix some terminology first. 

Let $E\rightarrow W$ be a real vector bundle of rank $4k$ over a manifold $W$. We say that $E$ is \emph{quasiquaternionic}, if a subbundle $\mathcal I\subset End(E)$ of real rank 3 admitting local trivializations $(I_1, I_2, I_3)$ with quaternionic relations%
\begin{equation}\label{Eq_QuatRelations}
 I_1I_2=-I_2I_1=I_3,\qquad I_1^2=I_2^2=I_3^2=-id
\end{equation}
is given. In this case $\mathcal I$ is called the structural bundle. Since any two trivializations of $\mathcal I$ as above differ by an $SO(3)$-gauge,  the structural bundle is naturally an oriented Euclidean vector bundle.

An example of quasiquaternionic bundle is any oriented Euclidean vector bundle $E$ of real rank 4. In this case we choose by default the structural bundle to be $\mathfrak{so}_+(E)\cong\Lambda^2_+E$. Another example is the tangent bundle of a quaternionic \kahler manifold. The vertical bundle of the nonlinear spinor bundle $\mathbb M$ as in Section~\ref{Subsect_ASDasPSW} is also quasiquaternionic.

Let $X$ be an arbitrary oriented four-manifold. Suppose $\rho\colon W\rightarrow X$ is a fibre bundle such that the vertical bundle $\mathcal V_\rho=\ker\rho_*$ is quasiquaternionic. Assume also that $W$ is equipped with a connection such that the horizontal bundle $\mathcal H_\rho$ is also quasiquaternionic. For instance, this is the case if $X$ is Riemannian.  

Suppose an isomorphism $\gamma\colon\mathcal I(\mathcal H_\rho)\rightarrow \mathcal I(\mathcal V_\rho)$ compatible with the Euclidean structures is given, i.e., for any local trivialization $(I_1, I_2, I_3)$ of $\mathcal I(\mathcal H_\rho)$ satisfying~(\ref{Eq_QuatRelations})  the triple  
$(J_1, J_2, J_3)=(\gamma(I_1), \gamma(I_2), \gamma(I_3))$ is a local trivialization of  $\mathcal I(\mathcal V_\rho)$ also satisfying~(\ref{Eq_QuatRelations}). Then the subbundles%
\begin{equation*}
   \begin{aligned}
     E_\rho^+ &=\set{A\in\hom R{\mathcal H_\rho}{\mathcal V_\rho}\; |\; I_1AJ_1+ I_2AJ_2+I_3AJ_3=A\;  },\\
    E_\rho^- &=\set{A\in\hom R{\mathcal H_\rho}{\mathcal V_\rho}\; |\; I_1AJ_1+ I_2AJ_2+I_3AJ_3=-3A\;  }
   \end{aligned}
\end{equation*}
of $\hom R{\mathcal H_\rho}{\mathcal V_\rho}$ are well defined and therefore we have the decomposition%
\begin{equation*}
  \mathcal H_\rho^*\otimes\mathcal V_\rho = E_\rho^+\oplus E_\rho^-.
\end{equation*}
Recall that for any $u\in\Gamma(X; W)$ the covariant derivative $\nabla u$ is a section of $T^*X\otimes u^*\mathcal V_\rho\cong u^*\duzhky{\mathcal H_\rho^*\otimes \mathcal V_\rho}$.

\begin{defn}
 The first order differential operator $\dirac$ defined by the sequence%
\begin{equation*}
 \Gamma(X; W)\xrightarrow{\ \nabla\ }\Gamma(T^*X\otimes u^*\mathcal V_\rho)\longrightarrow \Gamma(u^*E_\rho^-)
\end{equation*}
is called a (generalized) \emph{Dirac operator}.
\end{defn}

Notice that unlike in Definition~\ref{Defn_GenerDiracOp}, in the above definition the source manifold $X$ does not need to be equipped with a Riemannian structure.

\medskip

From now on we assume that $(W, g, \Omega)$ is a $Spin(7)$-manifold equipped with a Cayley fibration $\rho\colon W\rightarrow X$, where  $X$ is an arbitrary oriented four-manifold. Moreover, we also assume that $W$ is compact and $\rho$ has no critical points. The compactness of $W$ is assumed for the simplicity of exposition\footnote{the spinor bundle over a four-manifold provides a model for non-compact manifolds equipped with a Cayley fibration}, whereas the second assumption is an oversimplification. However, it is a necessary step before considering the general situation when some fibres are allowed to be singular.

For the Cayley fibration we define the horizontal bundle as the orthogonal complement of $\mathcal V_\rho$:%
\begin{equation}\label{Eq_SplittingTW}
 TW=\mathcal H_\rho\oplus \mathcal V_\rho.
\end{equation}
Observe that we have a distinguished isomorphism $\gamma\colon\mathfrak{so}_+(\mathcal H_\rho)\rightarrow \mathfrak{so}_+(\mathcal V_\rho)$. Indeed, let $Q(W)\rightarrow W$ be the $Spin(7)$-structure of $W$. Recall~\cite[Thm~1.38]{HarveyLawson:82} that at each point $w\in W$ the subgroup $K\subset Spin(7)$ that respects splitting~(\ref{Eq_SplittingTW}) is isomorphic to (\ref{Eq_GpK})  and  denote by $Q_\rho\rightarrow W$ the corresponding principal $K$-subbundle of $Q(W)$. Then $\mathcal H_\rho=Q_\rho\times_K E$ and $\mathcal V_\rho=Q_\rho\times_K F$ for some $K$-representations $E$ and $F$ such that $\Lambda^2_+E\cong\mathfrak{so}_+(3)\cong\Lambda^2_+ F$. Hence, we get the desired isomorphism~$\gamma$.

\begin{rem}
Each fibre $W_x$ of the Cayley fibration has  a hyperHermitian structure (defined up to an $SO(3)$-rotation). Indeed, pick a frame in $T_xX$. Then the horizontal lift combined with the Gram-Schmidt process defines a trivialization of $\mathcal H_\rho|_{W_x}$, so that $\mathfrak{so}_+(\mathcal H_\rho)$ also carries a trivialization. Finally, we equip  %
$W_x$ with a hyperHermitian structure via the map~$\gamma$.
\end{rem}

Further, similarly as in Section~\ref{Subsect_DiffFormsOnFibreBundles} the space of differential forms on $W$ is naturally equipped with the bigrading so that we have%
\begin{equation*}
   \begin{aligned}
      d = d^{1,0} + d^{0,1}+ d^{2,-1}\qquad\text{and }\qquad  \Om =\Om^{4,0}+\Om^{2,2}+\Om^{0,4}.
   \end{aligned}
\end{equation*}

For any $\e\in (0, 1]$ consider the metric $g_\e=g_h+ \e g_v$, where $g_h$ and $g_v$ are Euclidean scalar products on $\mathcal H_\rho$ and $\mathcal V_\rho$ respectively. The corresponding 4-form $\Om_\e$ is of comass 1 but in general it does not need to be closed. 

\begin{lem}
 For any $\e\in (0, 1)$ there exists a decomposition $\Om_\e=\Om_{1, \e}+ \Om_{2, \e}$ such that $\Om_{1, \e}$ is closed and $\Om_{2, \e}$ satisfies%
\begin{equation}\label{Eq_Om2e}
 -\om\wedge\om\wedge\Om_{2, \e} < |\om |^2 vol_W
\end{equation}
for any $\om\in\Om^2(W)$.
\end{lem}

\begin{proof}
 We have
\begin{equation*}
   \begin{aligned}
    \Om_\e &=\Om^{4,0}+\e\,\Om^{2,2}+\e^2\Om^{0,4}\\%
                  &=\e\,\Om +\duzhky{(1-\e)\Om^{4,0}-\e(1-\e)\Om^{0,4}}\\%
                  &=\Om_{1, \e}+ \Om_{2, \e}.
   \end{aligned}
\end{equation*}
By assumption $\Om_{1, \e}= \e\,\Om$  is closed. Further, for any 2-form $\om$ we have%
\begin{equation*}
    \begin{aligned}
     -\om\wedge\om\wedge\Om^{4,0}&=-\om^{0,2}\wedge\om^{0,2}\wedge \Om^{4,0} \\%
                                                                  &= \duzhky{\bigl |\om_-^{0,2}\bigr |^2-\bigl |\om_+^{0,2}\bigr |^2}vol_W\le |\om |^2 vol_W
    \end{aligned}
\end{equation*}
and similarly $\om\wedge\om\wedge\Om^{0,4}\le |\om|^2 vol_W$. Combining these inequalities, we obtain~(\ref{Eq_Om2e}).
\end{proof}

\begin{cor}[{\cite[Thm.\,6.1.3]{Tian:00}}]
 Let $G$ be a compact Lie group and $\mathbb P \xrightarrow{\ \eta\ } W$ be a principal $G$-bundle. Then for any $\e\in (0, 1]$ there exists a natural compactification of the moduli space $\EuScript M_{asd}^\e(\mathbb P)$ of asd connections with respect to the form $\Om_\e$.\qed
\end{cor}

Consider $\mathbb P$ as the fibre bundle over $X$ via the map $\tau\colon\mathbb P\xrightarrow{\ \eta\ }W\xrightarrow{\ \rho\ }X$. Then a connection $\phi$ on $\mathbb P\rightarrow X$ induces a connection $\varphi$ on $W\rightarrow X$. Indeed, think of a connection as a 1-form with values in the vertical bundle. Further, observe that $\mathcal V_\tau=\ker \eta_*\comp\rho_*=\eta_*^{-1}(\mathcal V_\rho)$. Then the connection $\varphi$ is determined via the requirement that the diagram%
\begin{displaymath}
 \begin{CD}
 \mathrm T\mathbb P  @> \eta_*> >  TW \\
 @V \phi VV                @VV\varphi V\\
 \mathcal V_\tau @> \eta_* > >  \mathcal V_\rho
 \end{CD}
\end{displaymath}
commutes. We assume a choice of connection $\phi$ inducing connection~(\ref{Eq_SplittingTW}) on $W$ is made.

\begin{rem}
 A connection $\phi$ as above does exist (but is not unique). Indeed, first notice that the space of all connections on $\mathbb P\rightarrow X$ inducing a given connection on $W\rightarrow X$ is convex. It is easy to check the existence of $\phi$ for trivial bundles. Then the existence of $\phi$ for nontrivial bundles can be obtained via gluing with the help of the partition of unity.

In the setup of Section~\ref{Subsect_DiffFormsOnFibreBundles}  $\phi$ was fixed via the lift of the $H$-action from $M$ to $P$ and the choice of a connection on the principal $H$-bundle $Q$. 
\end{rem}

Assume that for each $x\in X$ the moduli space $\mathcal M_{asd}(i_x^*\mathbb P)$ of asd-con\-nec\-tions on the fibre $W_x$ is nonsingular.  Denote by $\mathbb M_{asd}\rightarrow X$ the fibre bundle obtained by replacing $W_x$ by $\mathcal M_{asd}(i_x^*\mathbb P)$. Similarly, the fibre bundle $\mathcal A_{asd}(i_x^*\mathbb P)\rightarrow\mathcal M_{asd}(i_x^*\mathbb P)$ gives rise to the bundle $\mathbb A_{asd}\rightarrow\mathbb M_{asd}$ and the connection $\phi$ induces a connection on $\mathbb M_{asd}\rightarrow X$. Further, a hyperHermitian structure on $W_x$ induces a hyperHermitian structure on the corresponding fibre of $\mathbb M_{asd}$. In particular the vertical bundle of  $\mathbb M_{asd}$ is quasiquaternionic and there is also an induced isomorphism $\Gamma\colon\mathfrak{so}_+(\mathcal H_{\mathbb M_{asd}})\rightarrow \mathcal I(\mathcal V_{\mathbb M_{asd}})$.

Further, similarly as in the case of the spinor bundle we can write the asd equations with respect to the form $\Om_\e$ in the form%
\begin{equation}\label{Eq_ASDonCayley}
    \Pi_\e''\duzhky{F_{A_\e}^{2,0}+ F_{A_\e}^{0,2}}=0,\qquad
    \duzhky{F_{A_\e}^{1,1}}^+=0.
\end{equation}
The formal limiting form of system~(\ref{Eq_ASDonCayley}) as $\e\rightarrow 0$ is%
\begin{equation}\label{Eq_LimitingASDonCayley}
    \duzhky{F_{A_0}^{0,2}}^+=0,\qquad
    \duzhky{F_{A_0}^{1,1}}^+=0.
\end{equation}

 Let $\Gamma_0(\mathbb M_{asd})\subset \Gamma (\mathbb M_{asd})$ denote the subspace of all sections, which can be lifted to a section of $\mathbb A_{asd}$. 

\begin{thm}
 There exists a natural bijective correspondence between the moduli space of solutions to equations~(\ref{Eq_LimitingASDonCayley}) and the space $\mathcal H_0(\mathbb M_{asd})$ of  all harmonic spinors contained in $\Gamma_0(\mathbb M_{asd})$.\qed
\end{thm}

The proof of the above theorem can be obtained by a suitable modification of the proof of Theorem~\ref{Thm_LimitingPSW_HarmSpinors}. The  difference is that we can not interpret equations~(\ref{Eq_ASDonCayley}) as a \TP system, but it is easy to check directly that the arguments in the proof of Theorem~\ref{Thm_LimitingPSW_HarmSpinors} apply to the above statement as well. We omit the details.


\section{Concluding remarks}\label{Sect_ConcludingRemarks}

Let $M$ and $M'$ be two \hK manifolds endowed with actions of $Sp(1)\rtimes\mathcal G$ and $Sp(1)\rtimes\mathcal G'$ respectively such that $M_0=M\hkred\mathcal G$ and $M_0'=M'\hkred\mathcal G'$ are isomorphic as \emph{hypercomplex} manifolds (the Riemannian metrics are less important for what follows). Assume we are in a favourable situation when the spaces of solutions to the perturbed \TP equations~(\ref{Eq_PertPSW}) with targets $M$ and $M'$ are good approximations (in a suitable sense) of $\mathcal H_0(\mathbb M_0)\cong\mathcal H_0(\mathbb M_0')$. Then the \TP theories with target spaces $M$ and $M'$ are essentially equivalent.

\medskip

Recall that the ADHM construction represents the moduli space $\mathcal M_{n,k}$ of framed $SU(n)$-instantons of charge $k$ on $\mathbb R^4$ as the finite dimensional \hK reduction. In other words, a natural candidate for $M'$ in the context of $Spin(7)$-instantons on $\mathbb W^+\rightarrow X$ is the vector space 
\begin{equation*}
   M'\cong u(k)\otimes_{\mathbb R} W\,\oplus\,\mathbb C^n\otimes E\otimes W,
\end{equation*}
where $E$ denote the standard complex representation of $U(k)=\mathcal G'$. Notice that if $k=1$ we essentially arrive at the classical Seiberg-Witten theory. On the other hand, if we put formally $n=0$, which corresponds to the choice of %
$\mathfrak u(k)\otimes\mathbb H$ as a target manifold, then we get equations~(\ref{Eq_4DHitchinEqns}), i.e. a four-dimensional analogue of
Hitchin's theory~\cite{Hitchin:87}. In general, the choice of $M'$ as above leads to a mixture of both theories.

The problem is that the \hK reduction $M_0'$ is \emph{not} smooth (it is the Uhlenbeck compactification $\bar {\mathcal M}_{n,k}$ of $\mathcal M_{n,k}$) so that Theorem~\ref{Thm_LimitingPSW_HarmSpinors} is not applicable. One can partially overcome this difficulty as follows. Suppose $X$ is a \kahler surface so that we can modify slightly the original \TP equations:
\begin{equation*}
   \begin{cases}
      \dirac_b u=0, \\
      F_b^+ +\nu(u)=\xi\,\om_X,
   \end{cases}
\end{equation*}
where $\xi$ is a central element in $\mathfrak g'=\mathfrak u(k)$. Arguing along similar lines as in Section~\ref{Subsec_PerturbedTPeqns}, we arrive at the space of harmonic spinors (in fact, (anti)holomorphic sections) with the target $\mathcal M(n,k)=M'\hkred_{\mu=\xi}\mathcal G'$, which is smooth. In fact, there is~\cite{Nakajima:99} a natural holomorphic morphism\footnote{this morphism represents $\mathcal M(n,k)$ as a series of blow-ups of $ \bar {\mathcal M}_{n,k}$} $\mathcal M(n,k)\rightarrow \bar {\mathcal M}_{n,k}$ and hence the corresponding map between the spaces of holomorphic sections.

\medskip

We note in passing that it is also interesting to study the \TP gauge theories based on analogues of the ADHM construction for other types of \hK four-manifolds like tori~\cite{DonaldsonKronheimer:90} or ALE spaces~\cite{KronheimerNakajima:90}. Similarly, it is well-known that the moduli space of monopoles on $\mathbb R^3$ (the Atiyah-Hitchin manifold) can be constructed as infinite dimensional \hK reduction in two different ways~\cite{Hitchin:83}.  The author intends to continue his studies in the above directions.

\bigskip

\textsc{Acknowledgements.} I thank T.Walpuski and anonymous referees for helpful comments on an earlier version of this paper.



\end{document}